\documentclass{amsart}
\usepackage{amsfonts}
\usepackage{amsmath}
\usepackage{amssymb}
\usepackage{amsthm}

\newtheorem{theorem}{Theorem}
\newtheorem{prop}{Proposition}
\newtheorem{lemma}{Lemma}
\newtheorem{rem}{Remark}
\newtheorem{cor}{Corollary}
\newtheorem{exmp}{Example}

\begin{document}
\author{Mark Pankov}
\title{On extendability of permutations}
\address{Department of Mathematics and Computer Sciences, University of Warmia and Mazury, S{\l}oneczna 54, 10-710 Olsztyn, Poland}
\email{pankov@matman.uwm.edu.pl}

\maketitle

\begin{abstract}
Let $V$ be a left vector space over a division ring and let ${\mathcal P}(V)$ be the associated projective space.
We describe all finite subsets $X\subset V$ such that every permutation on $X$ can be extended to a linear automorphism of $V$
and all finite subsets ${\mathcal X}\subset {\mathcal P}(V)$
such that every permutation on ${\mathcal X}$ can be extended to an element of ${\rm PGL}(V)$.
Also, we reformulate the results in terms of linear and projective representations of symmetric groups.

\end{abstract}

\section{Introduction}
Let $V$ be a left vector space over a division ring $R$.
We suppose that $\dim V=n$ is finite and not less than $2$.
Denote by ${\mathcal P}(V)$ the associated projective space formed by $1$-dimensional subspaces of $V$.

Our first result (Theorem 1) is related to extendability of permutations on finite subsets of $V$ to
linear automorphisms of $V$:
if every permutation on a finite subset of $V$ can be extended to a linear automorphism of $V$ then
this subset is formed by linearly independent vectors or it consists of
$$x_{1},\dots,x_{m},-(x_{1}+\dots+x_{m}),$$
where $x_{1},\dots,x_{m}$ are linearly independent vectors.

Under the assumption that $R$ is a field, we determine all finite subsets ${\mathcal X}\subset {\mathcal P}(V)$
such that every permutation on ${\mathcal X}$ can be extended to an element of ${\rm PGL}(V)$.
Our second result (Theorem 2) states that there are precisely three distinct types of such subsets.

The main results (Theorems 1 and 2) will be reformulated in terms of linear and projective representations of symmetric groups
(Corollaries 1 and 2).

\section{Permutations on finite subsets of vector spaces}
Let $X$ be a finite subset of $V$ containing more than one vector.
Denote by $S(X)$ the group of all permutations on $X$.
We want to determine all cases when every element of ${\rm S}(X)$ can be extended to a linear automorphism of $V$.
This is possible, for example, if $X$ is formed by linearly independent vectors.

\begin{exmp}\label{exmp1}{\rm
Suppose that $X$ consists of linearly independent vectors $x_{1},\dots,x_{m}$ and
the vector
$$x_{m+1}=-(x_{1}+\dots+x_{m}).$$
For every $i\in \{1,\dots m-1\}$ we take any linear automorphism $u_{i}\in {\rm GL}(V)$ such that
$$u_{i}(x_{i})=x_{i+1},\;u_{i}(x_{i+1})=x_{i}\;\mbox{ and }\;u_{i}(x_{j})=x_{j}\;\mbox{ if }\;j\ne i,i+1,m+1.$$
Every $u_i$ sends $x_{m+1}$ to itself.
Consider a linear automorphism $v\in {\rm GL}(V)$ leaving fixed every $x_{i}$ for $i\le m-1$ and
transferring $x_{m}$ to $x_{m+1}$.
Then
$$v(x_{m+1})=-(v(x_{1})+\dots+v(x_{m}))=-(x_{1}+\dots+x_{m-1}-x_{1}-\dots-x_{m-1}-x_{m})=x_{m}.$$
So, all transpositions of type $(x_{i},x_{i+1})$ can be extended to linear automorphisms of $V$.
Since $S(X)$ is spanned by these transpositions,
every permutation on $X$ is extendable to a linear automorphism of $V$.
}\end{exmp}

\begin{theorem}\label{theorem1}
If every permutation on $X$ can be extended to a linear automorphism of $V$
then $X$ is formed by linearly independent vectors or it consists of
$$x_{1},\dots,x_{m},-(x_{1}+\dots+x_{m}),$$
where $x_{1},\dots,x_{m}$ are linearly independent.
\end{theorem}

\begin{proof}
Let $x_{1},\dots,x_{k}$ be the elements of $X$.
Suppose that these vectors are not linearly independent
and consider any maximal collection of linearly independent vectors from $X$.
We can assume that this collection is formed by $x_{1},\dots,x_{m}$, $m<k$.
Then every $x_{p}$ with $p>m$ is a linear combination of $x_{1},\dots,x_{m}$, i.e.
$x_{p}=\sum^{m}_{l=1}a_{l}x_{l}$.
Let $u\in {\rm GL}(V)$ be an extension of the transposition $(x_{i},x_{j})$, $i,j\le m$.
Then
$$u(x_{i})=x_{j},\;u(x_{j})=x_{i}\;\mbox{ and }\;u(x_{l})=x_{l}\;\mbox{ if }\;l\ne i,j.$$
We have
$$\sum^{m}_{l=1}a_{l}x_{l}=x_{p}=u(x_{p})=\sum^{m}_{l=1}b_{l}x_{l},\;\mbox{ where }\;
b_{i}=a_{j},\;b_{j}=a_{i}\;\mbox{ and }\;b_{l}=a_{l}\;\mbox{ if }\;l\ne i,j.$$
Since $x_{1},\dots,x_{m}$ are linearly independent, the latter means that $a_{i}=a_{j}$.
This equality holds for any $i,j\le m$ and we have
$$x_{p}=a(x_{1}+\dots+x_{m})$$
for some non-zero scalar $a\in R$.
Let $v\in {\rm GL}(V)$ be an extension of the transposition $(x_{1},x_{p})$.
Then
$$v(x_{1})=x_{p},\;v(x_{p})=x_{1}\;\mbox{ and }\;v(x_{i})=x_{i}\;\mbox{ if }\;i\ne 1,p$$
We have
$$x_{1}=v(x_{p})=a(v(x_{1})+\dots+v(x_{m}))=a(x_{p}+x_{2}+\dots+x_{m})=$$
$$=a^{2}(x_{1}+\dots+ x_{m})+a(x_{2}+\dots+x_{m})=a^{2}x_{1}+(a^{2}+a)(x_{2}+\dots+x_{m}).$$
Hence $a^{2}=1$ and $a^{2}+a=0$ which implies that $a=-1$ and
$$x_{p}=-(x_{1}+\dots+x_{m}).$$
This equality holds for every $p>m$.
Therefore, $k=m+1$ and the second possibility is realized.
\end{proof}

Let $X$ be a finite subset of $V$ such that every permutation on $X$ can be extended to a linear automorphism of $V$.
Suppose that $|X|\ge 2$ and $\langle X\rangle=V$.
Then for every $s\in S(X)$ there is the unique extension $\alpha_{X}(s)\in {\rm GL}(V)$.
If $s,t\in S(X)$ then $\alpha_{X}(st)$ and $\alpha_{X}(s)\alpha_{X}(t)$ both are extensions of $st$
which guarantees that
$$\alpha_{X}(st)=\alpha_{X}(s)\alpha_{X}(t).$$
Thus $\alpha_{X}$ is a monomorphism of $S(X)$ to ${\rm GL}(V)$
(it is clear that the kernel of $\alpha_{X}$ is trivial).
The image of $\alpha_{X}$ will be denoted $G(X)$.

\begin{cor}\label{cor1}
Let $G$ be a subgroup of ${\rm GL}(V)$ isomorphic to ${\rm S}_{m}$.
Let also $X$ be an orbit of $G$ such that $G$ acts faithfully on $X$ and $|X|=m$
\footnote{If $X$ is an orbit of $G$ and $G$ acts faithfully on $X$ then $|X|\ge m$.}. 
Suppose that there are not proper $G$-invariant subspaces of $V$.
Then the following assertions are fulfilled:
\begin{enumerate}
\item[$\bullet$]
$X$ is a base of $V$ or $X=\{x_{1},\dots,x_{n},-(x_{1}+\dots+x_{n})\}$,
where $x_{1},\dots,x_{n}$ form a base of $V$;
\item[$\bullet$]
$G=G(X)$.
\end{enumerate}
\end{cor}

\begin{proof}
Let $r$ be the homomorphism of $G$ to $S(X)$ transferring every $g\in G$ to $g|_{X}$.
Since the action of $G$ on $X$ is faithful, $r$ is a monomorphism.
It follows from our assumptions that
$G$ and $S(X)$ have the same number of elements.
Thus $r$ is an isomorphism which means that every permutation on $X$ can be extended to a linear automorphism of $V$.
Since $\langle X\rangle$ is $G$-invariant, we have $\langle X\rangle=V$.
Then $r^{-1}=\alpha_{X}$ and $G=G(X)$.
\end{proof}

\section{Permutations on finite subsets of projective spaces}

Let ${\mathcal X}$ be a finite subset of ${\mathcal P}(V)$ containing more than one element.
Denote by $S({\mathcal X})$ the group of all permutations on ${\mathcal X}$.
In this section we determine all cases when every element of $S({\mathcal X})$ can be extended to an element of ${\rm PGL}(V)$
(if $R$ is a field).

Recall that the group ${\rm PGL}(V)$ is formed by the transformations of ${\mathcal P}(V)$ induced by linear automorphisms of $V$.
Let $\pi$ be the natural homomorphism of ${\rm GL}(V)$ to ${\rm PGL}(V)$.
The kernel of $\pi$ consists of all homotheties $x\to ax$, where $a$ belongs to the center of $R$, i.e.
two linear automorphisms of $V$ induce the same element of ${\rm PGL}(V)$ if and only if one of them is a scalar multiple of the other.

We say that $P_{1},\dots,P_{m}\in {\mathcal P}(V)$ form an {\it independent} subset if
non-zero vectors
$x_{1}\in P_{1},\dots,x_{m}\in P_{m}$
are linearly independent.
Every permutation on an independent subset can be extended to an element of ${\rm PGL}(V)$.

Let $m\in \{2,\dots,n\}$.
An $(m+1)$-element subset ${\mathcal X}\subset {\mathcal P}(V)$ is called an $m$-{\it simplex}
if it is not independent and every $m$-element subset of ${\mathcal X}$ is independent.
For example, if $x_{1},\dots,x_{m}\in V$ are linearly independent and
$a_{1},\dots,a_{m}\in R$ are non-zero then
$$\langle x_{1}\rangle,\dots, \langle x_{m}\rangle\;\mbox{ and }\;\langle a_{1}x_{1}+\dots+a_{m}x_{m}\rangle$$
form an $m$-simplex.
Conversely, if $\{P_{1},\dots,P_{m+1}\}$ is an $m$-simplex then
there exist linearly independent vectors
$$x_{1}\in P_{1}\setminus\{0\},\dots,x_{m}\in P_{m}\setminus\{0\}\;
\mbox{ such that }\;
P_{m+1}=\langle x_{1}+\dots+x_{m}\rangle.$$
Every permutation on an $m$-simplex can be extended to an element of ${\rm PGL}(V)$
\cite[Section III.3, Proposition 1]{Baer}.

Following \cite[Section III.4, Remark 5]{Baer} we say that a subset ${\mathcal X}\subset {\mathcal P}(V)$ is {\it harmonic}
if there are linearly independent vectors $x,y\in V$ such that
$${\mathcal X}=\{\;\langle x \rangle,\langle y \rangle,\langle x+y\rangle,\langle x-y\rangle\;\}.$$

\begin{exmp}\label{exmp2}{\rm
Suppose that the characteristic of $R$ is equal to $3$ and ${\mathcal X}$ is the har\-mo\-nic subset
consisting of
$$P_{1}=\langle x \rangle,\;P_{2}=\langle y \rangle,\;P_{3}=\langle x+y\rangle,\;P_{4}=\langle x-y\rangle.$$
Consider $u_{1},u_{2},u_{3}\in {\rm GL}(V)$ satisfying the following conditions
$$\begin{array}{ll}
u_{1}(x)=y&u_{1}(y)=x,\\
u_{2}(x)=-x&u_{2}(y)=x+y,\\
u_{3}(x)=x&u_{3}(y)=-y.
\end{array}$$
Since the characteristic of $R$ is equal to $3$,  we have
$$u_{2}(x-y)=-x-(x+y)=-2x-y=x-y.$$
A direct verification shows  that every $\pi(u_{i})$
is an extension of the transposition $(P_{i},P_{i+1})$.
Since the group $S({\mathcal X})$ is spanned by all transpositions of type $(P_{i},P_{i+1})$,
every permutation on ${\mathcal X}$ can be extended to an element of  ${\rm PGL}(V)$.
}\end{exmp}

\begin{theorem}\label{theorem2}
Suppose that $R$ is a field.
If  every permutation on ${\mathcal X}$ can be extended to an element of ${\rm PGL}(V)$ then one of
the following possibilities is realized:
\begin{enumerate}
\item[$\bullet$] ${\mathcal X}$ is an independent subset;
\item[$\bullet$] ${\mathcal X}$ is an $m$-simplex, $m\in \{2,\dots,n\}$;
\item[$\bullet$] the characteristic of $R$ is equal to $3$ and ${\mathcal X}$ is a harmonic subset.
\end{enumerate}
\end{theorem}

\begin{lemma}\label{lemma1}
Suppose that $R$ is a field.
Let $f$ be an element of ${\rm PGL}(V)$ transferring $P\in {\mathcal P}(V)$ to $Q\in {\mathcal P}(V)$.
For any non-zero vectors $x\in P$ and $y\in Q$ there exists $u\in {\rm GL}(V)$
such that $\pi(u)=f$ and $u(x)=y$.
\end{lemma}

\begin{proof}
We take any $v\in {\rm GL}(V)$ such that $\pi(v)=f$. Then $v(x)=ay$ and the linear automorphism
$u:=a^{-1}v$ is as required.
\end{proof}

\begin{rem}{\rm
If $R$ is non-commutative then
a scalar multiple of a linear mapping is linear only in the case when the scalar belongs to the center of $R$.
}\end{rem}

\begin{proof}[Proof of Theorem \ref{theorem2}]
Let $P_{1},\dots,P_{k}$ be the elements of ${\mathcal X}$.
If ${\mathcal X}$ is not independent then we take any maximal independent subset in ${\mathcal X}$.
Suppose that it is formed by $P_{1},\dots,P_{m}$, $k<m$ and
consider $P_{p}$ with $p>m$. Every non-zero vector $y\in P_{p}$ is a linear combination of
non-zero vectors $y_{1}\in P_{1},\dots,y_{m}\in P_{m}$.
If this linear combination contains $y_{i}$ and does not contain $y_{j}$ for some $i,j\le m$ then
an element of ${\rm PGL}(V)$ extending the transposition $(P_{i},P_{j})$
does not leave fixed $P_{p}$ which is impossible.
This means that
$$y=a_{1}y_{1}+\dots+a_{m}y_{m},$$
where all $a_{1},\dots,a_{m}\in R$ are non-zero.

Thus $P_{1},\dots,P_{m}$ and $P_{p}$ form an $m$-simplex for every $p>m$.
If ${\mathcal X}$ consists of $m+1$ elements, i.e. $k=m+1$, then ${\mathcal X}$ is an $m$-simplex.
Consider the case when $k\ge m+2$.

We choose non-zero vectors $x_{1}\in P_{1},\dots,x_{m}\in P_{m}$ such that
$$x_{m+1}:=x_{1}+\dots+x_{m}\in P_{m+1}.$$
If $p\ge m+2$ then
$$P_{p}=\langle x_{p}\rangle,\;\mbox{ where }\; x_{p}= b_{1}x_{1}+\dots+b_{m}x_{m}$$
and all $b_{1},\dots,b_{m}\in R$ are non-zero.
Let $v$ be a linear  automorphism of $V$ such that $\pi(v)$ is an extension of the transposition $(P_{m+1},P_{p})$.
By Lemma \ref{lemma1}, we can suppose that $v$ sends $x_{m+1}$ to $x_{p}$.
Since $x_{1},\dots,x_{m}$ are linearly independent and $v(P_{i})=P_{i}$ for every $i\le m$,
the equality
$$v(x_{1})+\dots+ v(x_{m})=v(x_{m+1})=x_{p}=b_{1}x_{1}+\dots+b_{m}x_{m}$$
shows that $v(x_{i})=b_{i}x_{i}$ for every $i\le m$.
Then
$$v(x_{p})=b_{1}v(x_{1})+\dots+b_{m}v(x_{m})=b^{2}_{1}x_{1}+\dots+b^{2}_{m}x_{m}\in P_{m+1}$$
which means that $b^{2}_{1}=b^{2}_{2}=\dots=b^{2}_{m}$
and $b_{i}=\pm b_{j}$ for any $i,j\le m$. In other words,
$$x_{p}=b(\varepsilon_{1}x_{1}+\dots+\varepsilon_{m}x_{m}),$$
where $\varepsilon_{i}=\pm 1$ for every $i\in \{1,\dots,m\}$.
Since $x_{m+1}$ and $x_{p}$ are linearly independent, 
$\varepsilon_{i}\ne\varepsilon_{j}$ for some pairs $i,j\le m$. 
This guarantees that the characteristic of $R$ is not equal to $2$
and we can assume that
$$P_{p}=\langle x_{p}\rangle,\; \mbox{ where }\;
x_{p}=x_{1}+\dots+x_{q}-x_{q+1}-\dots-x_{m}
$$
and $1\le q<m$.

Now consider a linear automorphism $u\in {\rm GL}(V)$ such that $\pi(u)$ is an extension of the transposition $(P_{q},P_{q+1})$.
Then
$$u(P_{q})=P_{q+1},\;u(P_{q+1})=P_{q}\;\mbox{ and }\;u(P_{i})=P_{i}\;\mbox{ if }\;i\ne q,q+1.$$
By Lemma \ref{lemma1}, we suppose that $u$ leaves fixed $x_{m+1}$.
Since $x_{1},\dots,x_{m}$ are linearly independent,
the equality
$$u(x_{1})+\dots+ u(x_{m})=u(x_{m+1})=x_{m+1}=x_{1}+\dots+ x_{m}$$
implies that
$$u(x_{q})=x_{q+1},\;u(x_{q+1})=x_{q}\;\mbox{ and }\;u(x_{i})=x_{i}\;\mbox{ if }\;i\ne q,q+1,\;i\le m.$$
Then
$$u(x_{p})=u(x_{1})+\dots+u(x_{q})-u(x_{q+1})-\dots-u(x_{m})$$
belongs to $P_{p}$ only in the case when  $q=1$ and $m=2$, i.e.
$P_{p}=\langle x_{1}-x_{2}\rangle$ for every $p\ge 4$.
The latter means that ${\mathcal X}$ is the harmonic subset consisting of
$$P_{1}=\langle x_{1} \rangle,\;P_{2}=\langle x_{2} \rangle,\;P_{3}=\langle x_{1}+x_{2}\rangle,
\;P_{4}=\langle x_{1}-x_{2}\rangle.$$

Let $w$ be a linear automorphism of $V$ such that $\pi(w)$ is an extension of the transposition $(P_{1},P_{3})$
and $w(x_{1})=x_{1}+x_{2}$ (see Lemma \ref{lemma1}).
Since $w(P_{2})=P_{2}$ and $w(P_{3})=P_{1}$, we have
$$w(x_{1}+x_{2})=w(x_{1})+w(x_{2})=(x_{1}+x_{2})+cx_{2}\in P_{1}.$$
Then $c=-1$ and $w(x_{2})=-x_{2}$.
The equality $w(P_{4})=P_{4}$ implies that
$$w(x_{1}-x_{2})=w(x_{1})-w(x_{2})=(x_{1}+x_{2})+x_{2}=x_{1}+2x_{2}\in P_{4}.$$
Hence $x_{1}+2x_{2}=x_{1}-x_{2}$ and $2=-1$, i.e. the characteristic of $R$ is equal to $3$.
\end{proof}

Every representation $\alpha: S_{m}\to {\rm GL}(V)$ induces the projective representation
$\pi\alpha:S_{m}\to {\rm PGL}(V)$.
By  \cite{Schur}, there exist projective representations of symmetric groups
which are not induced by linear representations
(an explicit realization of such representations can be found in \cite{Nazarov}).
Now we establish an analogue of Corollary \ref{cor1} for projective representations of $S_{m}$.

Let ${\mathcal X}$ be a subset of ${\mathcal P}(V)$ such that
every permutation on ${\mathcal X}$ can be extended to an element of ${\rm PGL}(V)$.
Suppose that $|{\mathcal X}|\ge 2$ and there is not a proper subspace of $V$ containing every element of ${\mathcal X}$.
The following example shows that an extension of a permutation on ${\mathcal X}$
to an element of ${\rm PGL}(V)$ is not unique if ${\mathcal X}$ is a maximal independent subset
(an independent subset consisting of $n$ elements).

\begin{exmp}{\rm
Let $x_{1},\dots,x_{n}$ be a base of $V$
and let $a_{1},\dots,a_{n}$ be distinct non-zero scalars.
Consider the linear automorphism of $V$ transferring every $x_{i}$ to $a_{i}x_{i}$.
The associated element of ${\rm PGL}(V)$ is non-trivial, but it induces the identity permutation on the set consisting of
$\langle x_{1}\rangle ,\dots,\langle x_{n}\rangle$.
}\end{exmp}

\begin{prop}\label{prop1}
Let $\{P_{1},\dots,P_{n+1}\}$ and $\{P'_{1},\dots,P'_{n+1}\}$ be $n$-simplices in ${\mathcal P}(V)$.
The following two conditions are equivalent:
\begin{enumerate}
\item[$\bullet$] $R$ is a field,
\item[$\bullet$] there is a unique element of ${\rm PGL}(V)$ transferring every $P_{i}$ to $P'_{i}$.
\end{enumerate}
\end{prop}

\begin{proof}
See \cite[Section III.3]{Baer}.
\end{proof}

If $R$ is a field and ${\mathcal X}=\{P_{1},\dots,P_{n+1}\}$ is an $n$-simplex then, by Proposition \ref{prop1},
for every permutation $s\in S({\mathcal X})$ there is the unique extension $\overline{s}\in {\rm PGL}(V)$.
This correspondence is a monomorphism of $S({\mathcal X})$ to ${\rm PGL}(V)$.
Its image will be denoted by $G({\mathcal X})$.
Note that $G({\mathcal X})=\pi(G(X))$, where $X$ is formed by vectors
$$x_{1}\in P_{1},\dots,x_{n}\in P_{n}\;\mbox{ and }\;-(x_{1}+\dots+x_{n})\in P_{n+1}.$$

Suppose that $R$ is a field of characteristic $3$ and ${\mathcal X}$ is a harmonic subset.
Then every $3$-element subset of ${\mathcal X}$ is a $2$-simplex and
Proposition \ref{prop1} guarantees that
every permutation on ${\mathcal X}$ is uniquely extendable to an element of ${\rm PGL}(V)$.
As above, we get a monomorphism of $S({\mathcal X})$ to ${\rm PGL}(V)$
and denote its image by $G({\mathcal X})$.

\begin{cor}\label{cor2}
Let $R$ be a field and let $G$ be a subgroup of ${\rm PGL}(V)$ isomorphic to $S_{m}$.
Let also ${\mathcal X}$ be an orbit of $G$ such that $G$ acts faithfully on ${\mathcal X}$ and $|{\mathcal X}|=m$
\footnote{As in Corollary \ref{cor1}, if ${\mathcal X}$ is an orbit of $G$ and $G$ acts faithfully on ${\mathcal X}$ then $|{\mathcal X}|\ge m$.}.
Suppose that there are not proper $G$-invariant subspaces of $V$
\footnote{A subspace $S\subset V$ is $G$-invariant if every element of $G$ transfers ${\mathcal P}(S)$ to itself.}.
Then the following assertions are fulfilled:
\begin{enumerate}
\item[$\bullet$] ${\mathcal X}$ is a maximal independent subset or an $n$-simplex or
${\mathcal X}$ is a harmonic subset and the characteristic of $R$ is equal to $3$;
\item[$\bullet$]
if ${\mathcal X}$ is not independent then $G=G({\mathcal X})$.
\end{enumerate}
\end{cor}

The proof is similar to the proof of Corollary \ref{cor1}.

\end{document}